\documentclass[10pt]{amsart}

\usepackage{amsthm,amsmath,amssymb,epsfig,graphicx,latexsym,amsthm,mathdots}

\usepackage{graphicx,amsfonts,subfigure}

\usepackage{hyperref}
\numberwithin{equation}{section}

\usepackage[usenames, dvipsnames]{color}

\usepackage[ulem=normalem,draft]{changes}
\usepackage{cancel}

       \theoremstyle{plain}
\newtheorem{theorem}{Theorem}[section]
\newtheorem{lemma}[theorem]{Lemma}
\newtheorem{proposition}[theorem]{Proposition}

\theoremstyle{remark}
{
    \newtheorem{definition}[theorem]{Definition}
    
    \newtheorem{remark}[theorem]{Remark}

}

\def\dd{{\rm d}}
\newcommand{\ddt}{\frac{\rm d}{{\rm d} t} }

\def\weight(#1,#2){c_{#1,#2}}

\if {

}{{\caption}}
} \fi

\def\uh{\hat{u}}

\def\xh{\hat{x}}

\def\yb{\bar{y}}

\def\C{\mathcal{C}}

\def\I{\mathcal{I}}

\def\P{\mathcal{P}}

\def\SS{\mathcal{S}}

\def\eps{\varepsilon}

\def\Om{{\Omega}}

\def\1B{{\bf  1}}

\def\half{\mbox{$\frac{1}{2}$}}

\def\1B{{\bf  1}}

\newcommand{\RR}{\mathbb{R}}

\def\cR{\mathbb{R}}

\newcommand\be{\begin{equation}}
\newcommand\ee{\end{equation}}
\newcommand\ba{\begin{array}}
\newcommand\ea{\end{array}}
\newcommand{\bea}{\begin{eqnarray}}
\newcommand{\eea}{\end{eqnarray}}
\newcommand{\bean}{\begin{eqnarray*}}
\newcommand{\eean}{\end{eqnarray*}}

\def\rar{\rightarrow}

\def\ds{\displaystyle}

\newcommand{\umin}{ u_{\rm min}}
\newcommand{\umax}{ u_{\rm max}}
\def\benl{\begin{equation*}}
\def\eenl{\end{equation*}}

\newcommand{\dtt}{ \mathrm{d}t}

\newcommand{\tauh}{ \hat\tau}

\DeclareMathAlphabet{\mathpzc}{OT1}{pzc}{m}{it}

\begin{document}

\title[Shooting algorithm for state-constrained control-affine problems]{Well-posedness of the shooting algorithm for control-affine problems \\
with a scalar state constraint}

\author[M.S. Aronna and F. Bonnans and B.S. Goh]
{M. Soledad Aronna\address{Escola de Matem\'atica Aplicada FGV EMAp, 22250-900 Rio de Janeiro, Brazil}\email{soledad.aronna@fgv.br}
and
J. Fr\'{e}d\'{e}ric Bonnans\address{INRIA Saclay , L2S, CentraleSupelec, 91190 Gif-sur-Yvette, France}
\email{Frederic.Bonnans@inria.fr}  
and 
Bean San Goh\address{School of Electrical Engineering, Computing and Mathematical Sciences
Curtin University, Perth, Australia}\email {goh2optimum@gmail.com}
}

\thanks{The first author was supported by FAPERJ (Brazil) through the {\em Jovem Cientista do Nosso Estado} Program and by CNPq
  (Brazil) through the {\em Universal} Program and the
  Productivity Scholarship.
  The second author was supported by the FiME Lab Research Initiative (Institut Europlace de Finance), and by the PGMO program.
 }

\maketitle

\begin{abstract}
We deal with a control-affine problem with scalar control subject to bounds, a scalar state constraint and endpoint constraints of equality type.
For the numerical solution of this problem, we propose a shooting algorithm  and  provide a sufficient condition for its local convergence. 
We exhibit an example that illustrates the theory.
\end{abstract}


\section{Introduction}


In this article we deal with an optimal control problem governed by the dynamics
\benl
\dot x_t = f_0(x_t) + u_t f_1(x_t)\quad \text{for a.a. } t\in [0,T],
\eenl
subject to the control bounds
$$
u_{\rm min} \leq u_t \leq u_{\rm max},
$$
with $u_{\rm min} < u_{\rm max}$,
endpoint constraints like
$$
\Phi(x_0,x_T) = 0,
$$
and a scalar state constraint of the form
$$
g(x_t) \leq 0.
$$
For this class of problems,  we propose a shooting-like numerical scheme and we show a sufficient condition for its local quadratic convergence, that is also a second order sufficient condition for optimality (in a particular sense to be specified later on). Additionally, we solve an example of practical interest, for which we also prove optimality by applying  second order sufficient optimality conditions obtained in \cite{AronnaBonnansGoh2016}.

This investigation is strongly motivated by applications since we deal with both control and state constraints, which naturally  appear in realistic models. Many practical examples that are covered by our chosen framework can be found in the existing literature. A non exhaustive list includes the prey-predator model \cite{GLV74}, the Goddard problem in presence of a dynamic pressure limit \cite{SeywaldCliff1993,GraichenPetit2008}, the optimal control of the atmospheric arc for re-entry of a space shuttle seen in \cite{Bonnard2003}, an optimal production and maintenance system studied in \cite{MaurerKimVossen2005}, and a recent optimization problem on running strategies \cite{AftalionBonnans2014}. 
We refer also to \cite{dePinho2005}, \cite{MaurerDePinho2016}, \cite{Schaettler2006} and references therein.

As it is commonly known, the application of the necessary conditions provided by Pontryagin's Maximum Principle leads to an associated two-point boundary-value problem (TPBVP) for the optimal trajectory and its associated multiplier \cite{Vinter2000}. A natural way for solving TPBVPs numerically  is the application of {\em shooting algorithms} \cite{MorRilZan62}. This type of algorithms has been used extensively for the resolution of optimal control problems (see {\em e.g.} \cite{Bul71,BockPlitt1984,Pes94} and references therein). In particular, shooting methods have been applied to control-affine problems both with and without state constraints. Some works in this direction are mentioned in the sequel. Maurer \cite{Mau76} proposed a shooting scheme for solving a problem  with bang-singular solutions, which was generalized quite recently by Aronna, Bonnans and Martinon in \cite{AronnaBonnansMartinon2013}, where they provided a sufficient condition for its local convergence.
Both these articles \cite{Mau76} and \cite{AronnaBonnansMartinon2013} analyze the case with control bounds and no state constraints.
 Practical control-affine problems with state constraints were solved numerically in several articles, a non extensive list includes Maurer and Gillessen \cite{MauGil75}, Oberle \cite{Obe79}, Fraser-Andrews \cite{Fra89} {and the recent articles Cots \cite{Cots2017} and Cots {\em et al} \cite{CotsGergaud2022}}.
Up to our knowledge, there is no result in the existing literature concerning sufficient conditions for the convergence of shooting algorithms in the framework considered here.

The paper is organized as follows. In Sections \ref{SecProblem} and \ref{SecArcs} we introduce the problem and give the basic definitions. 
 A shooting-like method and a sufficient condition for its local quadratic convergence are given in Sections \ref{SecShooting} and \ref{SecSuff}, respectively. The algorithm is implemented in Section \ref{SecExamples} where we solve numerically a variation of the regulator problem and we prove the optimality of the solution analytically. 
 
 \vspace{4pt}

\noindent {\bf Notations.} 
Let $\cR^k$ denote the $k-$dimensional real space, {\em i.e.} the space of column real vectors of dimension $k,$ and by  $\cR^{k*}$ its corresponding dual space, which consists of $k-$dimensional row real vectors. With $\cR^k_+$ and $\cR^k_-$ we refer to the subsets of $\cR^k$ consisting of vectors with nonnegative, respectively nonpositive, components.
We write $h_t$ for the value of function $h$ at time $t$ if $h$ is a function that depends only on $t,$ and by $h_{i,t}$ the $i$th component of $h$ evaluated at $t.$
Let $h(t+)$ and $h(t-)$ be, respectively, the right and left limits of $h$ at $t,$ if they exist.
Partial derivatives of a function $h$ of $(t,x)$ are referred as $D_th$ or $\dot{h}$ for the derivative in time, and $D_xh,$ $h_x$ or $h'$ for the differentiation with respect to space variables. The same convention  is extended to higher order derivatives.
By $L^p(0,T)^k$ we mean the Lebesgue space with domain equal to the interval $[0,T]\subset \cR$ and with values in $\cR^k.$ The notations $W^{q,s}(0,T)^k$ and $H^1(0,T)^k$ refer to the Sobolev spaces (see Adams \cite{Ada75} for further details on Sobolev spaces).
We let $BV(0,T)$ be the set of functions with bounded total variation. 
In general, when there is no place for confusion, we omit the argument $(0,T)$ when referring to a space of functions. For instance, we write  {
$L^\infty$ for $L^\infty(0,T),$
}
or $(W^{1,\infty})^{k*}$ for the space of $W^{1,\infty}-$functions from $[0,T]$ to $\cR^{k*}.$
We say that a function $h: \cR^k \to \cR^d$ is of class $C^\ell$ if it is $\ell-$times continuously differentiable in its domain.

\section{The problem}\label{SecProblem}

Let us consider $L^\infty(0,T)$ and $W^{1,\infty}(0,T;\cR^n)$ as control and state spaces,  respectively. 
We say that a control-state pair $(u,x)\in L^\infty(0,T)\times W^{1,\infty}(0,T;\cR^n)$
is a  {\em trajectory} if it satisfies both the {\em state equation}
\be
\label{bsbstateeq}
\dot x_t = f_0(x_t) +u_{t} f_1(x_t) \quad \text{for a.a. } t\in [0,T],
\ee
and the finitely many  {\em endpoint constraints}
of equality type given by
\be
\Phi(x_0,x_T) = 0.
\ee
Here $f_0$ and $f_1$ are assumed to be Lipschitz continuous and twice continuously differentiable vector fields over $\RR^n$,
$\Phi$ is of class $C^2$ from $\RR^n\times\RR^n$ to $\RR^{q}.$
Under these hypotheses, for any pair control-initial condition $(u,x_0)$ in $L^\infty(0,T)\times \RR^n$, the state equation
\eqref{bsbstateeq} has a unique solution. Additionally, we consider
a {\em cost functional} 
$$
\phi:\cR^n \times \cR^n \to \cR,$$ 
the
 {\em control bounds} 
\be
\label{bsbcc} 
\umin\leq u_{t} \leq \umax \quad \text{for a.a. } t\in [0,T],
\ee
where $\umin < \umax$,
and a {\em scalar state constraint} 
\be
\label{stateconstraint1}
g(x_t) \leq 0\quad \text{for all } t\in [0,T],
\ee
with the functions $\phi$ and $g:\RR^n\rar \RR$ being of class $C^2.$ 
A trajectory $(u,x)$ is said to be {\em feasible} if it satisfies 
\eqref{bsbcc}-\eqref{stateconstraint1}.

\begin{remark}[On the control bounds]
We allow $u_{\min}$ and $u_{\max}$ to be either finite real numbers, or to take the values  $-\infty$ or $+\infty,$ respectively,  meaning that we also consider problems with control constraints of the form $u_t \leq u_{\max}$ or $u_{\min} \leq u_t$, as well as problems in the absence of control constraints.
\end{remark}

Summarizing, this article deals with the optimal control problem in the Mayer form given by
\be
\label{P}\tag{P}
\min \phi(x_0,x_T);
\qquad \text{subject to \eqref{bsbstateeq}-\eqref{stateconstraint1}.}
\ee

 \subsection{Types of minima}

Throughout this article, we make use of two notions of optimality that are {\em weak} and {\em Pontryagin minima} and are defined as follows.

\begin{definition}[Weak and Pontryagin minima]
\label{defminpoint}
A {\em weak minimum} for \eqref{P} is a feasible trajectory $(u,x)$ for which  there exists $\eps>0$ such that $\phi(x_0,x_T) \leq \phi(\tilde x_0,\tilde x_T)$ for  any feasible $( \tilde u,\tilde x)$ verifying $\|(\tilde u,\tilde x)-(u,x)\|_{\infty} < \eps.$ 

A feasible trajectory $(u,x)$ is called a {\em Pontryagin minimum} for \eqref{P} if for any $M>0,$ there exists $\eps_M>0$ such that $\phi(x_0,x_T) \leq \phi(\tilde x_0,\tilde x_T)$ for  any feasible $( \tilde u,\tilde x)$ satisfying 
\be
\label{defminpont1}
\| \tilde x - x\|_\infty +\| \tilde u -u \|_1 < \eps_M,\qquad \| \tilde u - u\|_\infty < M.
\ee
\end{definition}

 Note that any Pontryagin minimum is also a weak minimum. Consequently, necessary conditions that hold for weak minima, also do it for Pontryagin one. 
This article provides  a numerical scheme for approximating Pontryagin minima of \eqref{P}.  
 In order to achieve this, we make use of the auxiliary unconstrained transformed problem (TP) given in equations \eqref{costTP}-\eqref{continuityxTP}, which possesses neither control bounds {nor} state constraints and  can be solved  numerically in an efficient way.
In Lemma \ref{LemmaPontWeak} below we prove that transformed Pontryagin minima of  \eqref{P} that verify certain structural hypotheses are weak minima of the unconstrained transformed problem (TP).

\subsection{Bang, constrained and singular arcs}\label{SecArcs}
The {\em contact set} associated with the state constraint is defined as
\be
\label{defC}
C := \{ t\in [0,T]:\; g(\xh_t)=0 \}.
\ee
For $0 \leq a < b \leq T$, we say that $(a,b)$ is an 
{\em active arc} for the state constraint
or, shortly,  a {\em $C$ arc,} 
if $(a,b)$ is a maximal open interval contained in $C.$  
A point  $\tau\in (0,T)$ is a {\em junction point of the state constraint} 
if it is the extreme point of a $C$ arc.

Similar definitions hold for the control constraint, with the difference that the control variable is only almost everywhere defined.
The {\em contact sets for the control bounds} are given by
\begin{gather*}
B_- := \{ t\in [0,T]:\; \uh_t = \umin  \},
\quad B_+ := \{ t\in [0,T]:\; \uh_t = \umax  \},\\
 B:=B_- \cup B_+.
\end{gather*}
Note that these sets are defined up to null measure sets. Additionally, observe that if $\umin = - \infty$ then $B_- = \emptyset$ and, analogously, if $\umax=+\infty$ then $B_+ = \emptyset.$
We say that
$(a,b)$ is a {\em $B_-$} (resp. {\em $B_+)$ arc}
if $(a,b)$  is included, up to a null measure set, in
$B_-$ (resp. in $B_+$),
but no open interval strictly containing $(a,b)$ is.
We say that $(a,b)$ is a {\em $B$ arc} if it is either a
$B_{-}$ or a $B_+$ arc. 

Finally, let $S$ denote the {\em singular set} given by
\be
S:=\{ t\in [0,T]: u_{\min} < \uh_t < u_{\max} \text{ and } g(\xh_t) <0 \}.
\ee
We say that $(a,b)$ is an {\em $S$ arc} if $(a,b)$  is included,
up to a null measure set, in $S$, 
but no open interval strictly containing $(a,b)$ is.

We call {\em junction} or {\em switching times} the points $\tau \in
(0,T)$ at which the trajectory $(\xh,\uh)$ switches from one type of
arc ($B_-,B_+,C$ or $S$) to another type. Junction/switching times are
denominated by the type of arcs they separate. One can have, for instance, {\em $CS$ junction, $B_-B_+$ switching time, etc. } 

Throughout the remainder of the article, we assume that the {\em state constraint is of first order}, this is,
\be
\label{1order}
g'(\xh_t)f_1(\xh_t) \neq 0\quad \text{on } C,
\ee
and we impose the following {\em hypotheses on the control structure:}
\be
\label{chypradn}
\left\{ \ba{cl}
 {\rm (i)} &\text{the interval $[0,T]$ is (up to a zero measure set)
the disjoint}
\\ & \text{union of finitely   many arcs of type $B$, $C$ and $S,$
and the set $C$} \\
& \text{does not contain isolated points,}
\\
{\rm (ii)} &\text{the control $\uh$ is at uniformly positive distance
  of the bounds } 
\\& u_{\min} \text{ and } u_{\max}, \text{ over $C$ and $S$ arcs,}
\\ 
{\rm (iii)} &
\text{the control $\uh$ is discontinuous at CS and SC junctions.}
\ea \right.
\ee

{
\begin{remark}
Note that some problems 
(even if they are convex, like Fuller's problem \cite{MR0156744})
exhibit chattering phenomena with infinitely many switches,
necessarily with some very short arcs.
It is not clear how to deal with such problems with the method that we present in this article.
\end{remark}
}

The example of the regulator problem studied in Section
\ref{SecExamples} fullfils the above hypothesis \eqref{chypradn} (see
as well the example given in
\cite[Remark 2]{AronnaBonnansGoh2016}).


When a control satisfying hypothesis \eqref{chypradn}(i) is, for instance, a concatenation of a bang and a singular arc, we call it a {\em BS control.} This denomination is extended for any finite sequence of arc types.

In order to formulate our shooting algorithm, we express the control as a function of the state on $C$ arcs and we fix the control to its bounds on $B$ arcs.

\subsubsection{Expression of the control on constrained arcs}\label{SubsectionC}

From $g(\xh_t)=0$ on $C,$ we get
\be
 \label{dtg0}
 0 = \ddt g(\xh_t) = g'(\xh_t) \big( f_0(\xh_t) + \uh_t f_1(\xh_t) \big)\quad \text{on } C,
 \ee
 and, since \eqref{1order} holds,   we have that 
\be
 \label{uinC}
 \uh_t=  -\frac{g'(\xh_t)f_0(\xh_t)}{g'(\xh_t)f_1(\xh_t)}\quad \text{on } C.
 \ee

\section{Shooting formulation}\label{SecShooting}

We now explain how to state a transformed problem with neither control bounds nor running state constraints  that serves as an intermediate step to write a numerical scheme for problem \eqref{P}. Afterwards the optimality system of the transformed problem is reduced to a nonlinear equation in a finite dimensional space.

The starting point is to estimate the arc structure of the control, {\em  i.e.} the sequence of its different types of arcs and the approximate values of its junction times.
 This is done in practice by some {\em direct method} such as solving the nonlinear programming (NLP) associated to the discretization of the optimal control problem.
 Then we formulate a {\em transformed problem} in which the control is fixed to its bounds on B arcs, and is expressed as a function of the state on C arcs.  So the optimzation variables are now the control over singular arcs and the switching times. Subsequently, we express the optimality conditions of the transformed problem.  Finally, by eliminating the control as a function of the state and costate, we reduce the optimality system to a finite dimensional equation.

So,  let us assume for the remainder of the section that  $(\uh,\xh)$ 
is a Pontryagin minimum for \eqref{P}. 
Additionally, without loss of generality and for the sake of simplicity of notation, 
we set $u_{\min} :=0$ and $u_{\max} :=1.$  Recall further that $(\uh,\xh)$ complies with 
the structural hypotheses  \eqref{chypradn} for the control $\uh,$  
and that the state constraint is of first order,  {\em i.e.} \eqref{1order} holds true.

\subsection{The transformed problem}\label{SubsecReformulation}
We now state the transformed problem corresponding to \eqref{P}, in the
spirit of \cite{AronnaBonnansMartinon2013}, 
and we prove that any Pontryagin minimum for the original problem \eqref{P}
is transformed into a weak minimum of the unconstrained transformed problem.

For the Pontryagin minimum $(\uh,\xh)$, let 
\be
0=: \tauh_0 < \tauh_1 < \dots < \tauh_N := T
\ee
denote its associated switching times. Recall the definition of the sets $C,$ $B_-,$ $B_+$ and $S$ given in Section \ref{SecArcs} above. Set $\hat I_k := [\tauh_{k-1},\tauh_k]$  for $k=1,\dots,N,$ and
\be
\I(S) := \big\{{k \in \{1,\ldots ,N \}}: \hat I_k \text{ is a singular arc} \big\}.
\ee
Analogously, define $\I(C),$ $\I(B_-),$ and $\I(B_+).$ For each $k=1,\dots,N,$ consider a state variable $x^k \in W^{1,\infty}(0,T;\cR^n),$ and for each singular arc $k \in \I(S),$ a control variable $u^k \in L^\infty(0,T).$
On the set $B,$ we fix the control to the corresponding bound. Additionally, recall that from formula \eqref{uinC} we have that  $\uh_t=\Gamma(\xh_t)$ on $C,$ where $\Gamma$ is  given by 
$$
\Gamma(x):=  -\frac{g'(x)f_0(x)}{g'(x)f_1(x)}.
$$
After these considerations, we are ready to state the transformed problem. 
Define the optimal control problem (TP), on the time interval $[0,1],$ by
\begin{align}
& \label{costTP} \text{min } \phi(x^1_0,x^N_1), \\
& \label{eqxk1} \dot{x}^k = (\tau_k -\tau_{k-1}) \big(f_0(x^k) + u^k f_1(x^k)\big)\quad \text{for } k\in \I(S), \\
&  \dot{x}^k = (\tau_k -\tau_{k-1}) f_0(x^k)\quad \text{for } k\in  \I(B_-), \\
&  \dot{x}^k = (\tau_k - \tau_{k-1}) \big( f_0(x^k) + f_1(x^k) \big)\quad \text{for } k\in \I(B_+), \\
&  \label{eqxk4} \dot{x}^k = (\tau_k - \tau_{k-1}) \left(f_0(x^k) + \Gamma(x^k) f_1(x^k) \right)\quad \text{for } k\in \I(C), \\
& \dot\tau_k =0 \quad \text{for } k=1,\dots,N-1,\\
&\Phi(x_0^1,x_1^N) = 0,\\
& \label{gx0k} g(x_0^k) = 0 \quad  \text{for } k\in \I(C), \\
& \label{continuityxTP} x_1^k = x_0^{k+1} \quad \text{for } k=1,\dots,N-1.
\end{align}

\begin{remark}
Since we use the expression \eqref{uinC} to eliminate the control from the expression \eqref{dtg0} of the derivative of the state constraint equal to zero, we impose the entry conditions \eqref{gx0k} in the formulation of (TP)
 in order to guarantee that the state constraint is active along $x^k$ for every $k \in \I(C).$  
\end{remark}

Set 
\be
\label{changet}
\begin{split}
\xh^k_s &:= \xh \big(  \tauh_{k-1} + (\tauh_k - \tauh_{k-1})s \big) \quad \text{for}\ s\in [0,1] \text{ and } k=1,\dots,N,\\
\uh^k_s &:= \uh \big(  \tauh_{k-1} + (\tauh_k - \tauh_{k-1})s \big) \quad \text{for}\ s\in [0,1]  \text{ and }  k\in \I(S).
\end{split}
\ee

\begin{lemma}
\label{LemmaPontWeak}
Let $(\uh,\xh)$ be a Pontryagin minimum of problem \eqref{P}. Then 
the triple
$$
\Big(  (\uh^k)_{k \in \I(S)} , (\xh^k)_{k=1}^N,  (\tauh_k)_{k=1}^{N-1} \Big)
$$ 
is a weak solution of (TP).
\end{lemma}

\begin{proof}
Consider the feasible trajectories $\big( (u^k),(x^k),(\tau_k) \big)$  for (TP) satisfying
\begin{equation}
\label{estuk}
\| u^k - \uh^k \|_\infty < \bar\eps \quad \text{and} \quad   \lvert \tau_k - \hat{\tau}_k  \rvert  \leq \bar\delta \quad \text{for all } k=1,\dots,N,
\end{equation}
for some $\bar\eps,\bar\delta >0$ to be determined later. 
Set $I_k := [\tau_{k-1}, \tau_k]$ and consider the functions
$s_k \colon I_k \to [0,1]$ given by $s_{k,t} :=  \ds \frac{t-\tau_{k-1}}{\tau_k - \tau_{k-1}}.$ Define $u\colon[0,T] \to \cR$ by
\be
u_t :=
\left\{
\ba{cl}
\vspace{4pt} 0 & \text{if } t\in I_k,\, k\in \I(B_-), \\
\vspace{4pt}1 & \text{if } t\in I_k,\, k\in \I(B_+), \\
\vspace{4pt} \Gamma \big( x^k ( s_{k,t} ) \big) & \text{if } t\in I_k,\, k\in \I(C), \\
u^k(s_{k,t} ) & \text{if } t\in I_k,\, k\in \I(S).
\ea
\right.
\ee
Let $x \colon [0,T] \to \cR^n$ be the state corresponding to the control $u$ and the  initial condition $x(0)=x_0^1.$ 

We next  show that if $\bar\eps>0$ and $\bar\delta >0$ are small enough, then $(u,x)$ is  feasible for \eqref{P} and arbitrarily close to $(\uh,\xh)$ in the sense of
\eqref{defminpont1}. 
Observe that  $x(t) = x^k(s_{k,t})$ for all $k=1,\dots,N$ and $t\in I_k$. Hence, $x$ satisfies the endpoint constraints. Furthermore, due to Gronwall's Lemma, $(u,x)$ verifies the estimate
\be
\label{estux}
\|u-\uh\|_1 + \|x-\xh\|_\infty < \mathcal{O}(\bar\eps + \bar\delta).
\ee

Let us analyze the control constraints. Take $k=1,\dots,N.$ 
If $k\in \I(B_-) \cup \I(B_+),$ then $u_t \in \{0,1\}$ for a.a. $ t\in I_k.$ 
On the other hand, by the hypothesis \eqref{chypradn} on the control structure, there exists $\rho_1 > 0$ such that
\be
\label{bounduh}
\rho_1 < \uh_t < 1-\rho_1\quad \text{over  $C$ and $S$ arcs.}
\ee
  Suppose now that
$k\in \I(S).$ Then, in view of \eqref{estuk} and \eqref{bounduh}, we can see that
the control constraints hold on $I_k$
provided that $\bar\eps \leq \rho_1.$ Finally, let
$k$ be in $\I(C).$ Notice that in this case \eqref{bounduh} is equivalent to
\be
\rho_1 < \Gamma(\xh_t) < 1- \rho_1 \quad \text{ on } \hat{I}_k.
\ee
Hence, by standard continuity arguments and for $\bar\eps,\bar\delta$ sufficiently small, we get that
\be
0 < \Gamma (x_t) <1\quad \text{ on } I_k.
\ee 
We therefore confirm that $(u,x)$ verifies the control constraints.

Let us now consider the state constraint. Take first $k \in \I(C).$ Then $g(x_{\tau_k})= g(x_0^k)=0$ and, by definition of $(u,x),$ we have that $\ddt g(x_t) =0$ for all $t\in I_k.$ 
Therefore, $x$ satisfies the state constraint on $I_k$ for $k\in \I(C).$
Next, observe that, due to \eqref{chypradn}, for any $t\in [0,T]$ sufficiently far from a $C$ arc, one has $g(\xh_t) \leq -\rho$ for some small $\rho>0.$ 
Thus, by \eqref{estux} we get that $g(x_t)<0$ for appropriate $\bar\eps,\bar\delta.$ On the other hand, for $t\in [0,T]$ close to a $C$ arc, we reason as follows. Assume, without loss of generality, that $t$ is near an entry point $\tau_k$ of a C arc. In view of hypothesis \eqref{chypradn} and of the relation \eqref{dtg0}, we have that $ \frac{\rm d}{ {\rm d}  s} \rvert_{s=\hat\tau_k-} g(\xh_s) > 0,$ therefore, $ \frac{\rm d}{ {\rm d}  s} \rvert_{s=\tau_k-} g(x_s) > 0$ as well, if $\bar\eps,\bar\delta$ are sufficiently small. Consequently, $g(x_t) < 0.$ Hence, $x$ verifies the state constraint on $[0,T].$
With this, we conclude that $(u,x)$ is feasible for the original problem \eqref{P}.

Finally, given $M>0$ as in Definition \ref{defminpoint}, we can easily show that  $\bar\delta$ and $\bar\eps$ can be taken in such a way that $(u,x)$ satisfies \eqref{defminpont1} for such $M$ and the corresponding $\eps_M$ provided by the Pontryagin optimality of $(\uh,\xh).$ Consequently,
$
\phi(x_0,x_1) \geq \phi(\xh_0,\xh_1)
$
or, equivalently,
\be
\phi(x^1_0,x^N_1) \geq \phi(\xh^1_0,\xh^N_1),
\ee
which proves that $\big(  (\uh^k)_{k \in \I(S)},(\xh^k)_{k=1}^N,  (\tauh_k)_{k=1}^{N-1} \big)$ is a weak solution of (TP), as desired. This concludes the proof.
\end{proof}

\subsection{The shooting function}
We shall start by rewriting the problem (TP) in the following compact form, in order to ease the notation,
\begin{align}
\label{TPcost}  & \min\,\,  \tilde \phi(X_0,X_1),\\
\label{TPdyn} & \dot{X} = \tilde f_0(X) + \sum_{k\in \I(S)} U^k \tilde f_k(X),\\
\label{TPfinal} & \tilde \Phi (X_0,X_1)=0,
\end{align}
where $X:= \big( (x^k)_{k=1}^N,(\tau_k)_{k=1}^{N-1} \big),$ $U:=(u^k)_{k\in \I(S)},$ the vector field $ \tilde f_0: \cR^{Nn+N-1} \to  \cR^{Nn+N-1}$ is defined as follows, 
\benl
\begin{split}
\big( \tilde f_0(X) \big)&_{i=(k-1)n+1}^{kn}\\
&:=
\left\{
\ba{cl}
(\tau_k-\tau_{k-1})f_0(x^k) &\text{ for } k\in \I(S) \cup \I(B_-),\\
(\tau_k-\tau_{k-1}) \big(f_0(x^k)+f_1(x^k) \big) &\text{ for } k\in \I(B_+),\\
(\tau_k-\tau_{k-1}) \left( f_0(x^k)+\Gamma(x^k)  f_1(x^k) \right) &\text{ for } k\in \I(C),
\ea
\right.
\end{split}
\eenl
and $\big( \tilde f_0(X) \big)_{i=nN+1}^{Nn+N-1}:=0.$ Additionally,
for $k\in \I(S)$ the vector field $\tilde f_k: \cR^{Nn+N-1} \to  \cR^{Nn+N-1}$ is given by 
$$
\big( \tilde f_k (X) \big)_{i=(k-1)n+1}^{kn} := (\tau_k-\tau_{k-1})f_1(x^k),
$$
and $\big( \tilde f_k (X) \big)_{i}:=0$ for the remaining index $i,$
the new cost $\tilde\phi:\cR^{2(Nn+N-1)} \to \cR$ is  
$$
\tilde\phi (X_0,X_1)  :=\phi(x_0^1,x_1^N),
$$
and the function $\tilde \Phi: \cR^{2(Nn+N-1)} \to \cR^{ {\rm d}_{\tilde \Phi}}$ with ${\rm d}_{\tilde \Phi}:=q+\lvert \I(C) \rvert+n(N-1)$ is defined as
$$
\tilde\Phi(X_0,X_1):=
\begin{pmatrix}
\Phi(x_0^1,x^N_1) \\
\big( g(x_0^k) \big)_{k \in \I(C)} \\
\big( x_1^k-x_0^{k+1} \big)_{k=1}^{N-1}
\end{pmatrix}.
$$

The pre-Hamiltonian for problem (TP) is given by
\be
\tilde H =  P \Big(  \tilde{f}_0(X) +  \sum_{k\in \I(S)} U^k \tilde{f}_k(X) \Big) = \sum_{k=1}^N (\tau_k - \tau_{k-1})H^k,
\ee
where $P$ denotes the costate associated to (TP), \be
\label{Hk}
H^k:= p^k \big( f_0(x^k)+w^k f_1(x^k) \big),
\ee
with   the notation $w^k$ defined as
\be
\label{wk}
w^k := \left\{
\ba{cl}
u^k & \quad \text{if } k\in \I(S),\\
0 & \quad \text{if } k\in \I(B_-),\\
1 & \quad \text{if } k\in \I(B_+),\\
\Gamma(x^k) & \quad \text{if } k\in \I(C),
\ea
\right.
\ee
and $p^k$ denotes the $n$-dimensional vector of components $P_{(k-1)n+1},\dots,P_{kn}.$ Note  that  $w^k$ is a variable  only  for $k\in  \I(S)$, in  which case it represents the control $u^k$.

\subsection{Constraint qualification and first order optimality condition for (TP)}
Since problem (TP) has only endpoint equality constraints,
and the Hamiltonian is an affine function of the control, it is known that Pontryagin's Maximum Principle is equivalent to the first-order optimality conditions. So, the qualification condition is that the derivative of the constraint is onto at the nominal trajectory $(\hat U,\hat X)$, see {\em e.g.} \cite[Ch. 3]{PAOP}. 
This means that 
\be 
\begin{split}
\bar\Phi :  \,\cR^{Nn+N-1} \times (L^\infty)^{\lvert \I(S) \rvert}  \to \cR^{{\rm d}_{\tilde \Phi}}, 
\quad
 (X_0,U)  \mapsto \tilde \Phi(X_0,X_1),
\end{split}
\ee
where $X_t$ is the solution of \eqref{TPdyn} associated to $(X_0,U)$ is such that 
\be
\label{CQ2}
D\bar\Phi(\hat X_0,\hat U) \text{ is surjective.}
\ee
Under this hypothesis, the first-order optimality condition in normal form is as follows,
defining the endpoint Lagrangian associated to (TP) by:
\be
\tilde \ell^\Psi := \phi(x_0^1,x_1^N) + \sum_{j=1}^{q} \Psi_j \Phi_j(x_0^1,x_1^N) 
+ \sum_{k\in \I(C)} \gamma_k g(x_0^k) + \sum_{k=1}^{N-1} \theta_k(x_1^k-x_0^{k+1}).
\ee

\begin{theorem}
Let $(\hat U,\hat X)$ be a weak solution for (TP) satisfying the qualification condition \eqref{CQ2}. Then there exists a unique $\tilde \lambda := (\tilde \Psi,P) \in \cR^{{\rm d}_{\tilde\Phi}*}\times (W^{1,\infty})^{Nn+N-1*}$ such that $P$ is solution of 
\be
-\dot{P}_t  = D_X \tilde H(\hat U_t,\hat X_t,P_t) \quad \text{a.e. on } [0,T],
\ee
with {\em transversality conditions}
\be
\begin{split}
P_0 & = -D_{X_0} \tilde\ell^{\tilde \Psi} (\hat X_0,\hat X_1),\\
P_1 & = D_{X_T} \tilde \ell^{\tilde \Psi} (\hat X_0,\hat X_1),
\end{split}
\ee
and with
\be
\label{stationarity1}
\tilde H_U (\hat U_t,\hat X_t,P_t)=0.
\ee
\end{theorem}

{
\begin{proof}
This is a variant of
\cite[Theorem 1.174]{BonnansCSO}, where the cost function was supposed to be convex: one easily checks that the 
proof is essentially the same if the cost is
differentiable. See also \cite{ArutyunovKaramzin2020}.
\end{proof}
}

Since there is a unique associated multiplier, we omit from now on the dependence on $\tilde \lambda$ for the sake of simplicity of the presentation. Moreover, in some ocassions, we omit the dependence on the nominal solution $(\hat U,\hat X).$

\subsection{Expression of the singular controls in problem (TP)}\label{SubsectionSingular}

{ It is known that in this control-affine case, the control variable does not appear explicitly neither in the expression of $\tilde{H}_U$ nor in its time
derivative $\dot{\tilde{H}}_U$ (see {\em e.g.} \cite{Rob67,AronnaBonnansMartinon2013}).
The  {\em strengthened generalized Legendre-Clebsch condition} \cite{Rob67} for (TP) reads
\be
\label{LC}
-\frac{\partial}{\partial U } \ddot{\tilde H}_{U} \succ 0.
\ee
Here $A \succ B$, where $A$ and $B$ are symmetric matrices of same size, means that
$A-B$ is positive semidefinite.
At this point, recall the definitions of $H^k$ and $p^k$ given in \eqref{Hk} and in the first line after \eqref{wk}, respectively.
Simple calculations show that the l.h.s. of \eqref{LC} is a $\lvert \I(S) \rvert \times \lvert \I(S) \rvert$-diagonal matrix with positive entries equal to
\be
-(\tau_k-\tau_{k-1}) \frac{\partial}{\partial {u^k} } \ddot{H}^k_{u^k}\quad \text{for } k\in \I(S).
\ee
Then condition \eqref{LC} becomes 
\be
\label{LCa}
\frac{\partial}{\partial {u^k} }\ddot{H}^k_{u^k} < 0\quad \text{for } k\in \I(S).
\ee
Hence, thanks to \eqref{LCa}, for each $k\in \I(S)$ one can compute the control $u^k$ from the identity
\be
\label{ddotHuk}
\ddot{H}^k_{u^k}=0.
\ee
Apart from the previous equation \eqref{ddotHuk}, in order to ensure the stationarity $H^k_{u^k}=0,$ we add the following endpoint conditions:
\be 
0 = H^k_{u^k} (0) = p_0^k f_1(x_0^k),\quad 0=\dot{H}^k_{u^k}(0) = p_0^k [f_1,f_0](x_0^k)\quad \text{for } k\in \I(S).
\ee

\subsection{Lagrangians and costate equation}

The costate equation for $p^k$ is  
\be
\label{eqpk}
\dot{p}^k = -(\tau_k-\tau_{k+1}) D_{x^k}H^k,
\ee 
with endpoint conditions
\be
p_0^1=-D_{x_0^1} \tilde \ell^\Psi = -D_{x_0^1}\phi - \sum_{j=1}^{q} \Psi_j D_{x_0^1} \Phi_j - \chi_{\I(C)}(1) \gamma_1 g'(x_0^1),
\ee
\begin{gather}
\label{pk1} p_1^k = \theta^k \quad \text{for } k=1,\dots,N-1,\\
\label{pk0} p_0^k = \theta^{k-1} - \chi_{\I(C)}(k) \gamma_k g'(x_0^k) \quad \text{for } k=2,\dots,N,
\end{gather}
\be
p_1^N = D_{x_1^N} \phi + \sum_{j=1}^{q} \Psi_j D_{x_1^N} \Phi_j,
\ee
where $\chi_{\I(C)}$ denotes the {\em characteristic function} associated to the set $\I(C).$
For the costate $p^{\tau_k}$ we have the dynamics
\be
\label{pk}
\dot{p}^{\tau_k} = -H^k+H^{k+1},\quad  p_0^{\tau_k}=0,\  p_1^{\tau_k}=0\qquad \text{for } k=1,\dots,N-1.
\ee
It is known that the pre-Hamiltonian of autonomous problems has a constant value along an optimal solution (see {\em e.g.} 
\cite{Vinter2000}). By similar arguments it is easily seen that each $H^k$ is a constant function of time along an optimal solution.
Consequently, from \eqref{pk} we get that $p^{\tau_k}$ vanishes identically and that
\be
H_1^k = H_0^{k+1} \quad \text{for } k=1,\dots,N-1.
\ee

\subsection{The shooting function and method}

The shooting function associated with (TP) that we propose here is
\be
\begin{split}
 \SS : \cR^{Nn+N-1} \times \cR^{Nn+q+\lvert \I(C) \rvert,*} &\to \cR^{(N-1)n+N-1 +q+\lvert \I(C) \rvert+ 2\lvert \I(S) \rvert} \times \cR^{(N+1)n,*},\\
\big( (x_0^k),(\tau_k),(p_0^k),\Psi,\gamma  \big) &\mapsto
\begin{pmatrix}
\Phi(x_0^1,x_1^N) \\
\big( g(x_0^k) \big)_{k\in \I(C)} \\
(x^k_1 - x^{k+1}_0)_{k=1,\dots,N-1} \\
p^1_0 + D_{x_0^1} \tilde\ell^\Psi \\
p_1^k-p_0^{k+1}-\chi_{\I(C)}(k) \gamma_k g'(x_0^k) \\
p^q-D_{x^N_1} \tilde \ell^\Psi \\
(H^k_1-H^{k+1}_0)_{k=1,\dots,N-1} \\
\big( p^k_0 f_1(x^k_0) \big)_{k\in \I(S)} \\
\big( p_0^k [f_1,f_0](x^k_0) \big)_{k \in \I(S)}
\end{pmatrix},
\end{split}
\ee
where $\big( (x^k),(p^k) \big)$ is the solution of the state and costate equations \eqref{eqxk1}-\eqref{eqxk4}, \eqref{eqpk} with initial values $(x^k_0),(p^k_0),$ and control  $(u^k)_{k\in \I(S)},$  given by the stationarity condition \eqref{ddotHuk}. Note that we removed the variable $\theta$ by combining equations \eqref{pk1} and \eqref{pk0}.

The key feature of this procedure is that $\omega := \big( (x_0^k),(\tau_k),(p_0^k),\Psi,\gamma\big)$ satisfies 
\be
\label{S=0}
\SS(\omega)=0,
\ee
if and only if the associated solution $\big( (x^k),(p^k),(u^k) \big)$ verifies the Pontryagin's Maximum Principle for (TP). Briefly speaking, in order to find the candidate solutions of (TP), we shall solve \eqref{S=0}.

Let us observe that the system \eqref{S=0} has $2Nn+N-1+ q+\lvert \I(C) \rvert$ unknowns and $2Nn+N-1+ q +\lvert \I(C) \rvert+ 2\lvert \I(S) \rvert$ equations. Hence, as soon as a singular arc occurs, \eqref{S=0} has more equations than unknowns, {\em i.e.} it is overdetermined. 
We then follow \cite{AronnaBonnansMartinon2013}, where the authors suggested to solve the shooting equations by the Gauss-Newton method. 
We recall the following convergence result for Gauss-Newton, see  {\em e.g.} Fletcher \cite{fletcher2013practical}, or alternatively Bonnans \cite{bonnans2006numerical}.
If $F$ is a $C^1$ mapping from $\RR^n$ to $\RR^p$ with $p>n$, the Gauss-Newton method computes a sequence 
$(y^j)$ in $\RR^n$ satisfying $F(y^j)+DF(y^j)(y^{j+1}-y^j)=0$.
When $F$ has a zero at $\yb$ and $DF(\yb)$ is onto, the sequence  $(y^j)$ is well-defined provided that 
the starting point $y^0$ is close enough to $\yb$ and in  that case, $(y^j)$ converges superlinearly to $\yb$
(quadratically if $DF$ is Lipschitz near $\yb$).
  In view of the regularity hypotheses done in Section \ref{SecProblem}, we know that $\SS'$ is Lipschitz continuous.

\section{Sufficient condition for the convergence of the shooting algorithm}\label{SecSuff}
The main result of this article is Theorem \ref{ThConvergence} of current section. It gives a sufficient condition for the local convergence of the shooting algorithm, that is also a sufficient condition for weak optimality of problem (TP), as stated in Theorem \ref{SSC} below.

\subsection{Second order optimality conditions for (TP)} 

We now recall some second order optimality conditions for (TP). Let us consider the quadratic mapping on the space $ (L^\infty)^{\lvert \I(S) \rvert} \times (W^{1,\infty})^{Nn+N-1},$ defined as
\be
\tilde Q(V,Z) := \half D^2 \tilde\ell(Z_0,Z_1)^2 + \half \int_0^1 \big[ Z^\top  \tilde{H}_{XX} Z + 2 V \tilde H_{UX} Z \big] \dtt.
\ee
We next introduce the {\em critical cone} associated to (TP).
Since the problem has only qualified equality constraints, this critical cone coincides with the tangent space to the constraints.
Consider first the linearized state equation
\be
\label{LINEQ}
\dot Z = \tilde A Z + \tilde B V\quad \text{a.e. on } [0,1],
\ee
where $F(U,X):= \tilde f_0(X) + \sum_{k\in \I(S)} U^k \tilde f_k(X),$ $\tilde A := F_X,$ $\tilde B:= F_U;$
and let the linearization of the endpoint constraints be given by
\be
\label{LINCONS}
D\tilde \Phi (Z_0,Z_1) = 0.
\ee
The {\em critical cone}  for (TP) is defined as
\be
\tilde\C := \Big\{(V,Z) \in (L^\infty)^{\lvert \I(S) \rvert} \times (W^{1,\infty})^{Nn+N-1} : \text{\eqref{LINEQ}-\eqref{LINCONS} hold} \Big\}.
\ee
The following result follows (see {\em e.g.} \cite{LMO,ABDL12} for a proof).
\begin{theorem}[Second order necessary condition]
If $(\hat U,\hat X)$ is a weak minimum for (TP) that verifies \eqref{CQ2}, then
\be
\label{SONCineq}
\tilde Q(V,Z) \geq 0\quad \text{for all } (V,Z) \in \tilde\C.
\ee
\end{theorem}

In the sequel we present some optimality conditions for (TP). The first one is a necessary condition due to Goh \cite{Goh66} and the second one, a sufficient condition from Dmitruk \cite{Dmi77,Dmi87}. 
The idea behind these results lies on the following observation. Note that $\tilde{H}_{UU}$ vanishes and, therefore, the quadratic mapping $\tilde Q$ does not contain a quadratic term on the control variation $V.$ 
Consequently, the Legendre-Clebsch necessary optimality condition on the positive semidefiniteness of $\tilde{H}_{UU}$ holds trivially and  a second order sufficient condition cannot be obtained by  strengthening inequality \eqref{SONCineq}. 
In order to overcome this issue and derive necessary conditions for this singular case, Goh introduced a change of variables in \cite{Goh66a} and
applied it to derive necessary conditions in \cite{Goh66}. Some years later, Dmitruk \cite{Dmi77} showed a second order sufficient condition in terms of the coercivity of the transformation  $\tilde \Omega$ of $\tilde Q$ introduced below.

The {\em Goh transformation} for the linear system \eqref{LINEQ} is given by
\be
\label{GOH}
Y_t: = \int_0^t V_s {\rm d} s,\qquad \Xi_t:= Z_t - \tilde B_t Y_t.
\ee 
Notice that if $(V,Z) \in \tilde \C,$ then $(Y,\Xi)$ defined by the above transformation \eqref{GOH} 
satisfies (removing time indexes):
\begin{equation}
\dot  \Xi = A Z + \tilde B U - \tilde B U - \dot{ \tilde B} Y
=
A (\Xi + \tilde B Y) - \dot{ \tilde B} Y,
\end{equation}
and, therefore $\Xi$ 
is solution of the {\em transformed linearized equation} 
\be
\label{LINEQGOH}
\dot\Xi = \tilde A \Xi + \tilde{E} Y,
\ee
{where $\tilde E := \tilde A \tilde B - \ddt \tilde B,$} and $\Xi$ satisfies the {\em transformed linearized endpoint constraints}
\be
\label{LINCONSGOH}
D\tilde \Phi(\Xi_0,\Xi_1+\tilde B_1 h)=0,
\ee
where we set $h:= Y_1.$

Consider the function
\be
\rho(\zeta_0,\zeta_1,h) := D^2 \tilde\ell (\zeta_0,\zeta_1+\tilde B_1 h)^2 + h\tilde H_{UX,1} (2\zeta_1+\tilde B_1 h),
\ee
and the quadratic mapping
\be
\tilde\Omega (Y,\tilde h,\Xi) := \half \rho(\Xi_0,\Xi_1,\tilde h) + \half \int_0^1 \left( \Xi^\top \tilde{H}_{XX} \Xi + 2Y \tilde M \Xi + Y \tilde{R} Y\right) \dtt,
\ee
for $(Y,\tilde h,\Xi) \in  (L^2)^{\lvert \I(S) \rvert} \times \cR \times (H^1)^{Nn+N-1}$ and
\be
\tilde M:= \tilde{f}_1^\top \tilde{H}_{XX}  - \ddt \tilde{H}_{UX} -\tilde{H}_{UX} \tilde{A},
\quad
\tilde R:= \tilde{f}_1^\top \tilde{H}_{XX} \tilde{f}_1-2 \tilde{H}_{UX}\tilde{E}- \ddt (\tilde{H}_{UX} \tilde{f}_1).
\ee

Let us recall that the second order necessary condition for optimality stated by Goh \cite{Goh66} (and nowadays known as {\em Goh condition}) implies that if $(\hat U,\hat X)$ is a weak minimum for (TP) verifying \eqref{CQ2}, then 
\be
\label{CB}
\tilde{H}_{UX}\tilde B \text{ is symmetric,}
\ee
or, equivalently, $P\cdotp  D\tilde f_i  \tilde f_j= P \cdotp D \tilde f_j \tilde f_i$ for all $i,j=1,\dots,N.$ In the recent literature, this condition can be encountered as $P \cdotp [\tilde f_i,\tilde f_j]=0.$ 
We shall mention that this necessary condition was first stated by Goh in \cite{Goh66} for the case with neither control nor state constraints, and extended in \cite{ABDL12,FraTon13} for problems containing control constraints. 

Notice that when the control  variable of (TP) is scalar ({\em i.e.} $\lvert \I(S) \rvert=1$), then \eqref{CB} is trivially verified since $\tilde{H}_{UX}\tilde B$ is also a scalar. 
\if{
Furthermore, given the special structure of the dynamics of (TP), we can see that $D \tilde f_i\, \tilde f_j =0$ for all $i,j \in \{1,\dots,N\}$ with $ i\neq j $ and, consequently, the matrix in \eqref{CB} is diagonal and the Goh condition holds trivially for (TP) even when $\lvert \I(S) \rvert>1$. 
} \fi

{
We claim that $( Df_i ) f_j=0$, for all $i,j 
\in \{1, \ldots ,N\}$ with $i\neq j$ and, consequently,
the matrix in (68) is diagonal and the Goh condition holds trivially for (TP) even when ${\mathcal I}(S) >1$.
Indeed, let $k$, $\ell$ be disjoint elements of 
${\mathcal I}(S)$.
By the definition, the nonzero components of 
$\tilde f_k(X)$ have coordinates in the set
$I_k := \{ (k-1)n+1, kn\}$.
So, only the rows of 
$D \tilde f_k(X)$ in $I_k$ can be nonzero.
However, since
$I_k\cap I_\ell$ is empty,
$\tilde f_\ell$ has only zero 
components in $I_k$. 
Our claim follows.
}

From \eqref{CB} and \cite[Theorem 4.4]{ABDL12} we get the following result:

\begin{proposition}
For all $(V,Z) \in (L^\infty)^{\lvert \I(S) \rvert} \times (W^{1,\infty})^{Nn+N-1}$ solution of \eqref{LINEQ} and $(Y,\Xi)$ given by the Goh transform \eqref{GOH}, it holds
\benl
\tilde Q (V,Z) = \tilde \Omega (Y,Y_T,\Xi).
\eenl
\end{proposition}

Define, for $(\Xi_0,Y,\tilde h) \in \cR^n \times  (L^2)^{\lvert \I(S) \rvert} \times \cR,$ the {\em order function}
\be
\gamma(\Xi_0,Y,\tilde h) := \lvert \Xi_0 \rvert^2 + \int_0^1 \lvert Y_t \rvert^2 \dtt + \lvert  \tilde h \rvert ^2.
\ee

\begin{definition}[$\gamma$-growth]
A feasible trajectory $(\hat U,\hat X)$ of (TP) satisfies the {\em
  $\gamma$-growth condition in the weak sense} if there exists a
positive constant $c$ such that, for every sequence of feasible
variations $\big\{ (\delta X^k_0,V^k) \big\}_k$ converging to 0 in
$\cR^{Nn+N-1} \times ( L^\infty)^{\lvert \I(S) \rvert}$,
one has that
\be
\tilde \phi(X_0^k,X_1^k) - \tilde \phi(\hat X_0,\hat X_1) \geq c \gamma(\Xi^k_0,Y^k,Y^k_T),
\ee
for $k$ large enough, where $(Y^k,\Xi^k)$ are given by Goh transform \eqref{GOH} and $X^k$ is the solution of the state equation \eqref{TPdyn} associated to $(\hat{X}_0+\delta X_0^k,\hat{U}+V^k).$
\end{definition} 

Consider the {\em transformed critical cone}
\be
\tilde\P^2_S := \left\{ (Y,\tilde h,\Xi) \in(L^2)^{\lvert \I(S) \rvert} \times \cR \times (H^1)^{Nn+N-1}:  \text{\eqref{LINEQGOH}-\eqref{LINCONSGOH} hold} \right\}.
\ee

The following characterization of $\gamma$-growth holds (see \cite{Dmi77} or \cite[Theorem 3.1]{Dmi87} for a proof).

\begin{theorem}
\label{SSC}
Let $(\hat U,\hat X)$ be such that the qualification condition \eqref{CQ2} holds. Then $(\hat U,\hat X)$ is a weak minimum of (TP) that satisfies $\gamma$-growth in the weak sense if and only if \eqref{CB} holds and there exists $c>0$ such that
\be
\label{POS}
\tilde \Omega(Y,\tilde h,\Xi) \geq  c \gamma (\Xi_0,Y,\tilde h) \quad \text{on } \tilde\P_S^2.
\ee
\end{theorem}

We are now ready to state the following convergence result for the  shooting algorithm.
\begin{theorem}
\label{ThConvergence}
If $(\hat U,\hat X)$ is a weak minimum of problem (TP) satisfying the constraint qualification \eqref{CQ2} and the uniform positivity condition \eqref{POS}, then the shooting algorithm is locally quadratically convergent.
\end{theorem}

\begin{proof}
This is a consequence of the convergence result in \cite[Theorem 5.4]{AronnaBonnansMartinon2013}.
\end{proof}

Note that in the proof of the above Theorem, it is established that the hypotheses imply that the derivative of the shooting function is injective.

\section{Application to a regulator problem}\label{SecExamples}

Consider the following regulator problem, where $T=5$:
\be
\begin{split}
&\min \half \int_0^5 (x_{1,t}^2 + x_{2,t}^2) \dd t + \half x_{1,5}^2, \\
&  \dot x_{1,t} = x_{2,t},\quad \dot x_{2,t} =   u_t \in [-1,1],
\end{split}
\ee
subject to the state constraint and initial conditions
\be
x_{2,t} \geq -0.2, \quad x_{1,0}=0,\quad x_{2,0}=1.
\ee
To write the problem in the Mayer form, we introduce an auxiliary state variable given by the dynamics
\be 
\dot x_{3,t} =  \half  (x_{1,t}^2 + x_{2,t}^2),\quad x_{3,0}=0.
\ee
The resulting problem is then
\be
\label{Preg}
\begin{split}
& \min\, x_{3,5} + \half x_{1,5}^2, \\
 & \dot x_1 = x_2, \quad  \dot x_2 =   u, \quad \dot x_3 =  \half  (x_1^2 + x_2^2), \\
 &x_{1,0}=0,\quad x_{2,0}=1,\quad x_{3,0}=0,\\
 &-1 \leq u \leq 1,\; \; x_2 \geq -0.2. 
\end{split}
\ee

Using the optimal control solver BOCOP  \cite{Bocop2017}  we estimated that the optimal control $\uh$ is a concatenation of a bang arc in the lower bound, followed by a constrained arc and ended with a singular one. Briefly, we can say that the optimal control has a  $B_-CS$ structure.

\subsection{Checking local optimality}

 In this subsection we compute analytically the optimal solution of \eqref{Preg} and check that it verifies the second order sufficient condition for state-constrained control-affine problems proved in \cite[Theorem 5]{AronnaBonnansGoh2016}.
While the problem is convex, and hence, satisfying the first order optimality conditions is enough for proving optimality, the quoted second order conditions are of interest since they imply the quadratic growth, see 
\cite[Definition 3]{AronnaBonnansGoh2016}.

To problem \eqref{Preg} we associate the functions $g:\cR^3 \to \cR,$ $f_0,f_1:\cR^3 \to \cR^3$ given by
\benl
g(x):= -x_2-0.2,\quad f_0(x):= \begin{pmatrix} x_2\\ 0 \\ \half( x_1^2+x_2^2 )\end{pmatrix},\quad
f_1(x) = \begin{pmatrix} 0\\ 1\\ 0\end{pmatrix}.
\eenl
So that the  optimal  control $\uh$ is equal to -1 on the $B_-$ arc, and to 
$\ds-\frac{g'(\xh)f_0(\xh)}{g'(\xh)f_1(\xh)} = 0$
on  $C$  according to formula \eqref{uinC}, where $\xh$ is the associated  optimal state.

 The pre-Hamiltonian of \eqref{Preg} reads 
$$
H(u,x,p):= p_1 x_2 + p_2u +\half p_3(x_1^2+x_2^2),
$$
where $p:=(p_1,p_2,p_3)$ is the costate.


 Over the singular arc $S$, the inequality 
$u_{\min} < \uh_t < u_{\max}$ and  the minimum condition
of Pontryagin's Maximum Principle (see {\em e.g.} \cite[equation (2.12)]{AronnaBonnansGoh2016}) imply that 
\be
\label{stationarity2}
H_u=0.
\ee
Differentiating in time, (see {\em e.g.}
\cite{Mau76,AronnaBonnansMartinon2013}),
we obtain
\be
H_u = pf_1,\quad \dot{H}_u = p[f_1,f_0],\quad \ddot{H}_u
=   p\big[[f_1,f_0],f_0 \big] + \uh p\big[[f_1,f_0],f_1\big],
\ee
where $[X,Y]:= X'  Y - Y' X$ denotes the {\em Lie bracket}
associated with a pair of vector fields $X,Y : \cR^n \to \cR^n.$
In this control-affine case, the control variable does not appear
explicitly neither in the expression of $H_u$ nor in its time
derivative $\dot{H}_u.$
So, if
$p\big[[f_1,f_0],f_1\big]$ takes only nonzero values along singular arcs,
we obtain an expression for the control on singular arcs, namely
\be
\label{uinS}
\uh=-\frac{p\big[[f_1,f_0],f_0 \big](\xh)}{p\big[[f_1,f_0],f_1\big](\xh)}.
\ee



The involved Lie brackets for this examples are
\be
\label{Lie}
[f_1,f_0]= \begin{pmatrix}  -1 \\ 0 \\ -x_2 \end{pmatrix},\quad
\big[ [f_1,f_0],f_0 \big] =  \begin{pmatrix} 0 \\ 0 \\ x_1 \end{pmatrix},\quad
\big[ [f_1,f_0],f_1 \big]= \begin{pmatrix} 0 \\ 0\\ -1 \end{pmatrix}.
\ee
On the other hand, the costate equation on the singular arc $S$ gives
\be 
\label{examplep}
\begin{split}
\dot{p}_1 &= - p_3\xh_1,\quad p_{1,5} = \xh_{1,5}, \\
\dd{p}_2 &= -(p_1+p_3 \xh_2)\dd t+\nu, \quad p_{2,5}=0, \\
\dot{p}_3 &= 0,\quad p_{3,5}=1,
\end{split}
\ee
{where $\mu$ is the multiplier associated with the state constraint, and 
$\nu$ is the density of $\mu$.}
Thus, 
\be\label{examplep3}
p_3 \equiv 1,
\ee
 and from \eqref{uinS} and \eqref{Lie} we get
\be
\label{uonS}
\uh=\xh_1\quad \text{on}\ S. 
\ee
Moreover, the first order optimality conditions imply that
\be
\label{p2CS}
0=H_u=p_2\quad \text{on } C\cup S.
\ee

Let us write $\hat{\tau}_1,\hat{\tau}_2$ for the switching times, so that  
$$
B_-=[0,\hat{\tau}_1],\quad C=[\hat{\tau}_1,\hat{\tau}_2],\quad S=[\hat{\tau}_2,5].
$$ 
Since the control $\uh$ is constantly equal to $-1$ on $B_-,$ then 
\be
\label{examplex2}
\xh_{2,t} = 1-t \quad \text{on. } B_-,
\ee
until it saturates the state constraint at $\hat{\tau}_1.$ Hence $1- \hat{\tau}_1 = -0.2,$ so it follows that 
\be
\label{hattau1}
\hat\tau_1=1.2
\ee
and $B_- = [0,1.2].$ Consequently, 
\be\label{examplex1}
\xh_{1,t} = t-\frac{t^2}{2} \quad \text{on } B_-.
\ee
On $C=[1.2,\hat{\tau}_2],$ necessarily $\xh_2 \equiv -0.2,$ thus $u\equiv 0$ and 
\be
\label{x1C}
\xh_{1,t} = 0.48 - 0.2(t-1.2) = 0.72 - t/5 \quad \text{ on } C.
\ee
On the singular arc $S=[\hat{\tau}_2,5]$ we get, from the expression of $\uh$ on $S$ \eqref{uonS}, that
$\ddot \xh_1 = \dot \xh_2 = \uh = \xh_1.
$
Thus
\be
x_{1,t} = c_1 e^{5-t} + c_2 e^{t-5}\quad \text{ on } S,
\ee
for some real constant values $c_1,c_2.$
Therefore,
\be 
{\xh}_{2,t} = \dot{\xh}_{1,t} = -c_1 e^{5-t} + c_2 e^{t-5} \quad \text{ on } S.
\ee
The stationarity condition  $0=H_u=p_2$ on $S$ yields $0=\dot{p}_2 = -p_1-\xh_2,$ where the second equality of latter equation follows from  \eqref{examplep} and since $\nu \equiv 0$ on $S.$
Thus $p_1=-\xh_2$ on $S$, so
\be
p_{1,t} = c_1 e^{5-t} - c_2 e^{t-5}  \quad \text{ on } S.
\ee
The transversality condition for $p_1$ (see \eqref{examplep}) implies
$c_1+c_2 = c_1-c_2.$
Thus, $c_2=0$ and 
\be
\xh_1=-\xh_2 = p_1 = c_1 e^{5-t}  \quad \text{ on } S.
\ee
Since $-0.2=x_{2,\hat{\tau}_2},$ then 
$
x_{1,\hat{\tau}_2} = 0.2.
$
Hence, from the expression of $\xh_1$ on $C$ given in \eqref{x1C}, we obtain $0.72-\hat{\tau}_2/5 = 0.2$ and, consequently
\be
\label{hattau2}
\hat{\tau}_2=2.6.
\ee
Additionally, 
$-0.2 = x_{2,\hat{\tau}_2} = -c_1 e^{2.4}$ so that 
$
c_1 = 0.2 e^{-2.4}. 
$
At time $\hat{\tau}_2$ we have that 
$p_1= c_1 e^{2.4} = 0.2$
and, from the costate equation \eqref{examplep} and from \eqref{x1C}, we have 
\be
\label{regp1-dc}
\dot p_{1,t} = - \xh_{1,t} = t/5 - 0.72 
\leq 2.6/5 -0.72 < 0 \quad \text{on } C = [1.2,2.6].
\ee
So $p_1$ is decreasing on $C$ and,  recalling \eqref{p2CS}, this implies that 
\be
\nu_t = p_{1,t} + \xh_{2,t} > p_{1,\hat{\tau}_2}- 0.2 =0\quad \text{for } t\in [\hat{\tau}_1,\hat{\tau}_2).
\ee
Thus the complementarity condition for the state constraint \cite[equation(3.4)(ii)]{AronnaBonnansGoh2016} holds true. 

In order to check that the strict complementarity hypothesis (i) of \cite[Theorem 5]{AronnaBonnansGoh2016} is satisfied, we will prove that the corresponding strict complementarity condition for the control constraint of  \cite[Definition 4]{AronnaBonnansGoh2016} holds.
In view of \eqref{regp1-dc}, we get
\be
p_{1,t} = 0.2 + (t-\hat{\tau}_2)^2/10 - 0.72(t-\hat{\tau}_2). 
\ee 
Therefore 
$
p_{1,\hat{\tau}_1} = 0.2 + (1.4)^2/10 +0.72 \times 1.4 = 1.404.
$
On the $B_-$ arc,
we have seen that $\xh_1>0$ due to \eqref{examplex1} and $x_{1,0}=0,$ so
that $\dot p_1 = - \xh_1 <0$, and since 
$p_{1,\hat{\tau}_1}>0,$ it follows that $p_1$ has values greater than
$p_{1,\hat{\tau}_1} = 1.404$.
Therefore, since $\xh_2 >-0.2$ over $B_-$,
\be
\dot{p}_2=-p_1-\xh_2 < 0\quad \text{ on } B_-.
\ee
Since $p_{2,\hat{\tau}_1}=0,$ we get $p_2 > 0$ on $[0,\hat{\tau}_1)$ or, equivalently, $H_u>0$ on $[0,\hat{\tau}_1).$ We conclude that the strict complementarity condition for the control constraint given in {\cite[Definition 4]{AronnaBonnansGoh2016}} holds. This completes the verification of \cite[Theorem 5 - (i)]{AronnaBonnansGoh2016}.

Let us now verify the uniform positivity \cite[equation (4.6)]{AronnaBonnansGoh2016}.
The dynamics for the linearized state is 
\be
\begin{split}
& \dot{z}_1 = z_2,\quad \dot{z}_2=v,\quad \dot{z}_3 = \hat{x}_1 z_1 + \hat{x}_2 z_2,\\
& z_{1,0} = z_{2,0} = z_{3,0} =0.
\end{split}
\ee
Let $\C_S$ and ${\P}^2_*$ denote the strict critical cone and the extended cone (invoking the notation used in \cite{AronnaBonnansGoh2016}) at the optimal trajectory $(\uh,\xh),$ respectively. Since $\uh$ is $B_-CS$ then, for any $(v,z) \in \C_S,$  $v=0$ on the initial interval $B_-.$ Consequently, 
\be
\label{yB-}
y=0 \quad \text{on } B_-,\text{ for all } (y,\xi,h) \in \P^2_*.
\ee
On the other hand, the dynamics for the transformed linearized state is
\begin{gather}
\label{eqxi}
\dot\xi_1 = \xi_2+y,\quad \dot\xi_2=0,\quad \dot\xi_3 = \xh_1 \xi_1 + \xh_2\xi_2 + \xh_2y,\\
\xi_{1,0} = \xi_{2,0} = \xi_{3,0}=0.
\end{gather}

Take $(y,\xi,h) \in \P^2_*.$ Then
\be
\label{xi2}
\xi_2\equiv 0\quad \text{on } [0,5].
\ee
In view of \cite[equation (3.15)(i)]{AronnaBonnansGoh2016}, we have that $-\xi_2-y=0$ on $C.$ Thus, due to \eqref{xi2}, we get that
\be
\label{y=0}
y=0\quad \text{on } B_- \cup C.
\ee
Thus, from \eqref{eqxi}-\eqref{y=0}, we get $\xi_1=\xi_3=0$ on $B_- \cup C.$ 
{Regarding the last component $h$ of the considered critical direction $(y,\xi,h) \in \P^2_*,$ in view of the linearized
cost equation \cite[equation (3.16)]{AronnaBonnansGoh2016} and due to the fact that there are no final constraints, we get that}
\be
\label{h1}
\xh_{1,5}\,\xi_{1,5} + \xi_{3,5} = 0.
\ee
 Then, we deduce that there is no restriction on $h.$
We obtain
\be
\P^2_* = 
\left\{
\begin{split}
& (y,\xi,h) \in L^2 \times (H^1)^n \times \cR : \,\, y=\xi_1=\xi_2 = 0 \text{ on } B_- \cup C,\\
& \xi_2 = 0 \text{ on } [0,5],\ \dot\xi_1 = y \text{ and } \dot \xi_3 = \xh_1\xi_1 + \xh_2 y\ \text{ on } S, \, \eqref{h1} \text{ holds}
\end{split}
\right\}.
\ee
The quadratic forms $Q$ and $\Omega$ are given by
\be
Q(v,z):= \int_0^T( z^2_1 +z^2_2) \dd t + z^2_{1,5}; 
\quad 
\Om(y,h,\xi):= \int_0^T( \xi^2_1 +
(\xi_2+y)^2) \dd t + h^2.
\ee
Thus, $\Om$ is a Legendre form on $\{(y,h,\xi) \in L^2\times \cR
\times (H^1)^3:\text{\eqref{eqxi} holds}\}$ and is coercive on
$\P^2_*.$ Hence Theorem 5 in \cite{AronnaBonnansGoh2016}
holds. Consequently, $(\uh,\xh)$ is a Pontryagin minimum of problem
\eqref{Preg} satisfying the $\gamma$-growth condition.

\subsection{Transformed problem}
 Here we transform the problem \eqref{Preg} to obtain a problem with neither control nor state constraints, as done in Section \ref{SubsecReformulation}. 
 
The optimal control associated with \eqref{Preg} has a $B_-CS$ structure, as said above. Then we  triplicate the number of state variables obtaining the new variables $X_1,\dots,X_9$ and we 
 consider two switching times that we write $X_{10},X_{11}.$ The new
 problem has only one control variable that corresponds to the
 singular arc and which we call $V.$  The reformulation of
 \eqref{Preg} is as follows 
\be
\label{Pregref}
\begin{split}
\min\  &X_{9,1} + \half X^2_{7,1}, \\
& \dot{X}_1 = X_{10} X_2, \\
& \dot{X}_2 = - X_{10}, \\
& \dot{X}_3 = \half X_{10} (X_1^2 + X_2^2), \\
& \dot{X}_4 = (X_{11}-X_{10})X_5, \\
& \dot{X}_5 = 0, \\
& \dot{X}_6 = \half  (X_{11}-X_{10}) (X_4^2 + X_5^2), \\
& \dot{X}_7 = (5-X_{11}) X_8,\\
& \dot{X}_8 =(5-X_{11}) V,\\
& \dot{X}_9 = \half (5-X_{11})  (X_7^2 +X_8^2),\\
& \dot{X}_{10} = 0, \\
& \dot{X}_{11} = 0,\\
& X_{1,0}=0,\quad X_{2,0}=1,\quad X_{3,0}=0, \\
& X_{1,1}=X_{4,0},\quad X_{2,1} = X_{5,0},\quad X_{3,1} = X_{6,0},\\
& X_{4,1}=X_{7,0},\quad X_{5,1} = X_{8,0},\quad X_{6,1} = X_{9,0},\\
& X_{2,1}=0.2, \, {( \text{or } g(x_2(\tau_2))=0)}. \\
\end{split}
\ee
From \eqref{uonS} we deduce that $V=X_7.$

We solved problem \eqref{Pregref} by the shooting algorithm proposed in Section \ref{SecShooting}. The graphics of the optimal control and states are shown in Figure \ref{regulator} in the original variables $u$ and $x,$ and the corresponding costate variables $p_1,p_2,p_3$ are displayed in Figure \ref{regulator2}.
 
 
In our numerical tests we take 1000 time steps.
The optimal switching times obtained numerically are
$\hat{\tau}_1=1.2$ and $\hat{\tau}_2 = 2.6036023,$ so these values
agree with the ones founded analytically
in \eqref{hattau1} and \eqref{hattau2},
while the optimal cost is $0.3934884$ and the shooting function evaluated at the optimal trajectory is $5.\times 10^{-6}.$

\begin{figure}[!h]
\centering
\begin{center}
\subfigure{\includegraphics [width=5.5cm]{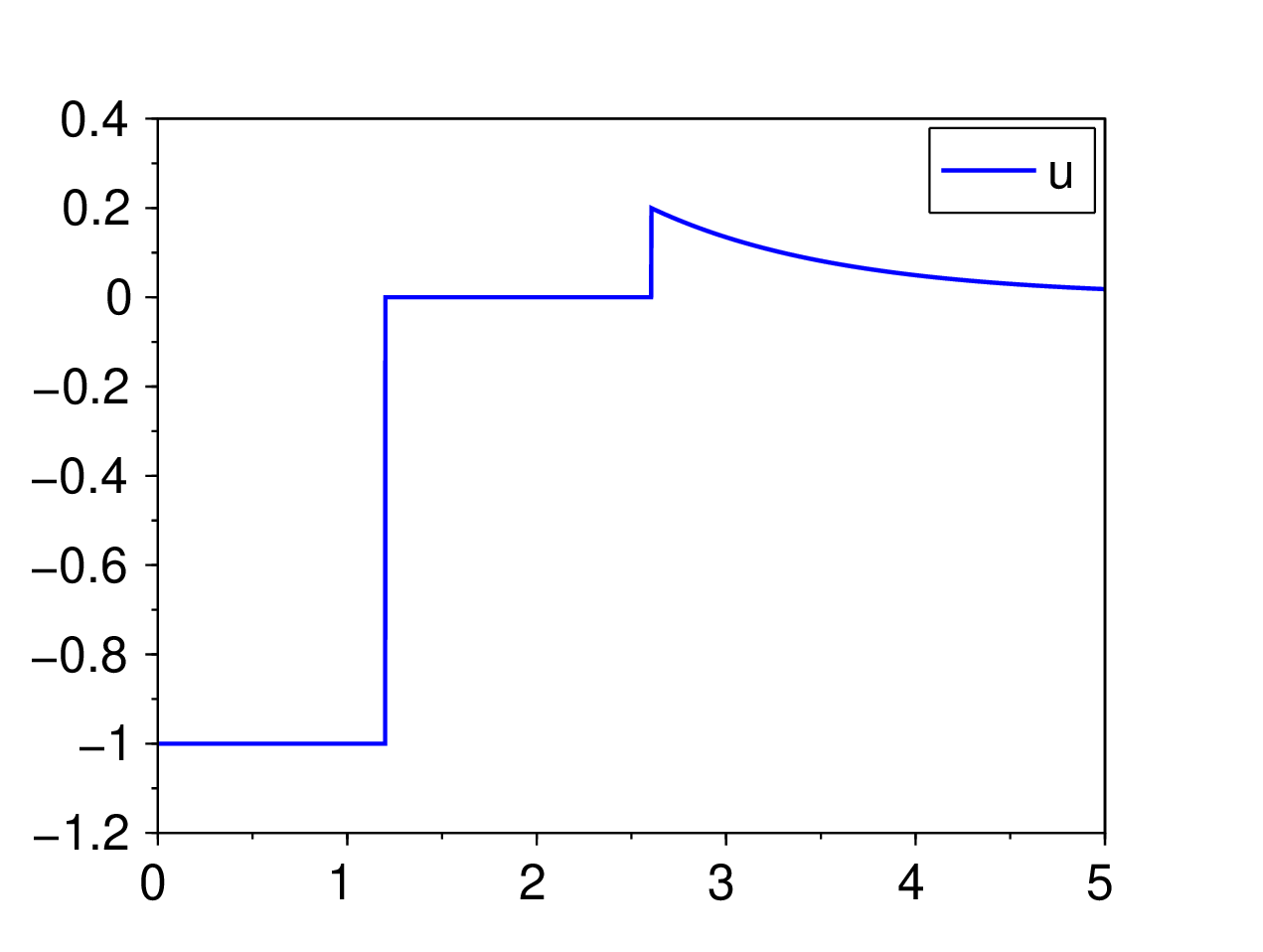}}
\subfigure{\includegraphics [width=5.5cm]{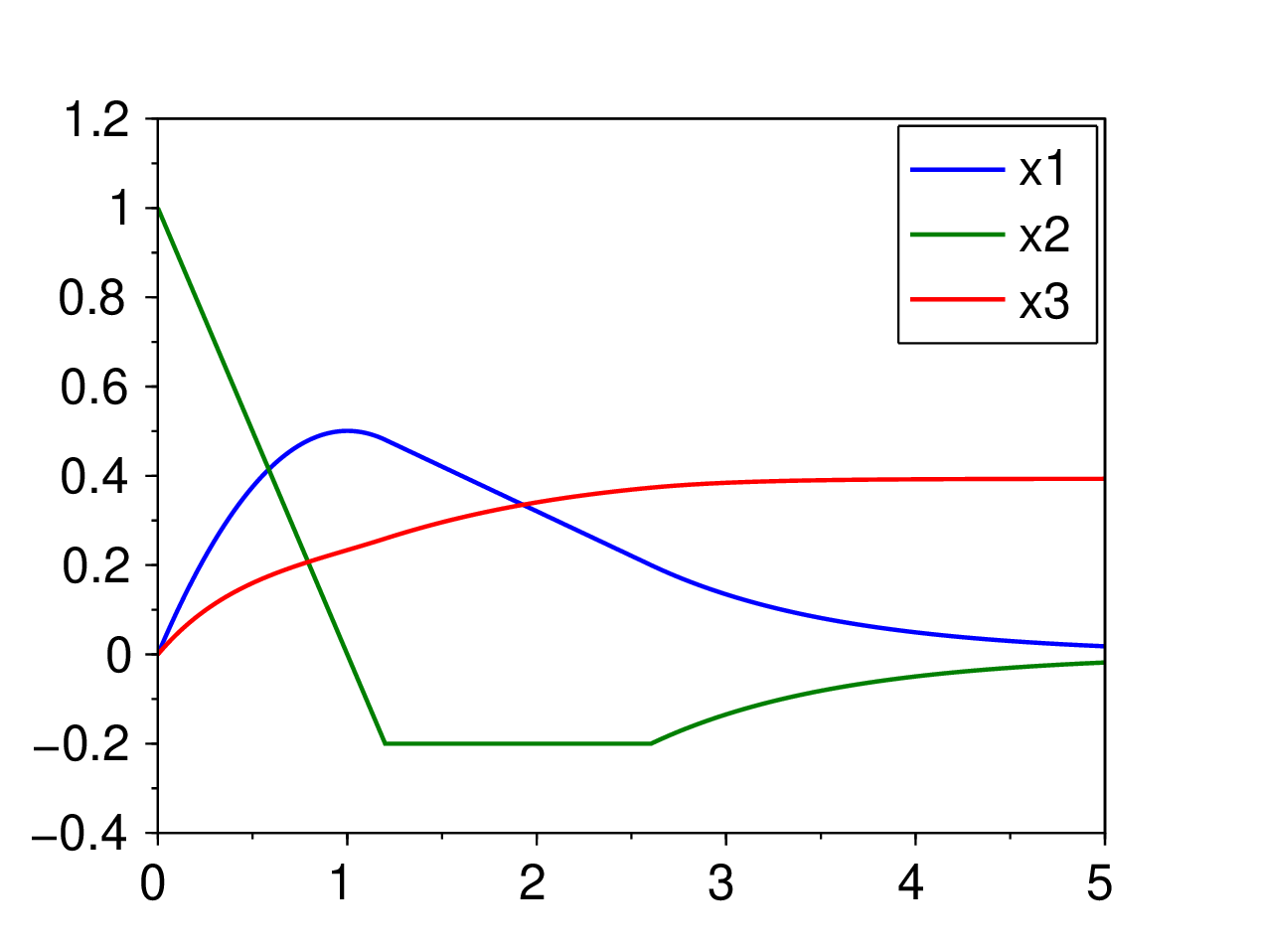}}
\end{center}
\caption{Optimal control and state variables}
\label{regulator}
\end{figure}

\begin{figure}[!h]
\centering
\begin{center}
\includegraphics [width=5.5cm]{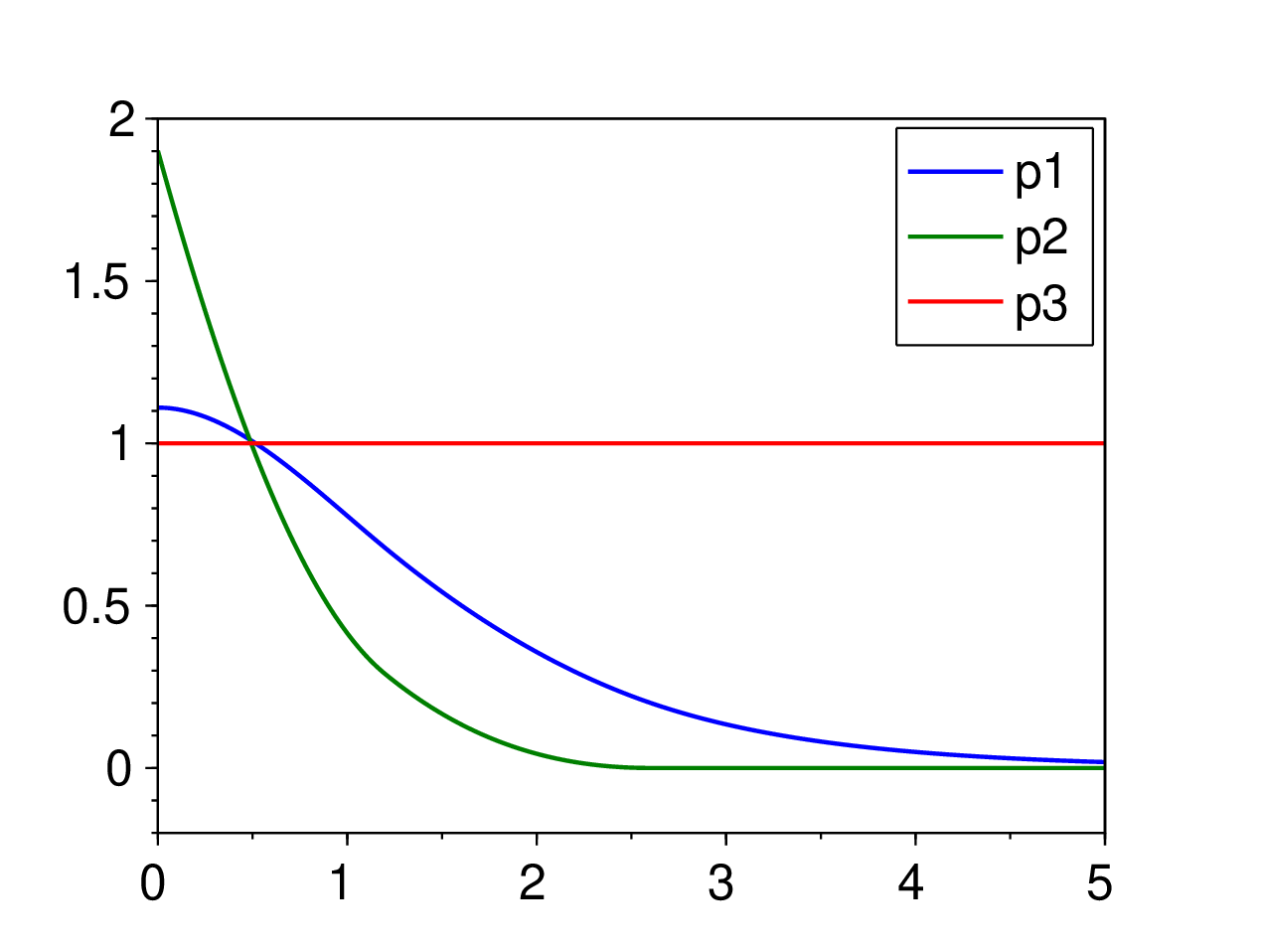}
\end{center}
\caption{Costate variables}
\label{regulator2}
\end{figure}

{
\begin{remark}
In order to compare with existing methods, we have solved the same problem using GEKKO Python \cite{beal2018gekko}. We have tested different variable definitions for the control input and we found the best result by setting the control as a {\em manipulated variables} \cite{beal2018gekko}.
The results are shown in Figure \ref{GEKKO2}. 
On the right of this figure,  we exhibit a zoom around the second switching point $\hat \tau_2$ of the optimal control. We can see that the method shows a  discontinuity of the control at $\hat \tau_2,$ but that the approximation of $\hat \tau_2$ is between $2.57$ and $2.58$, while we have calculated analytically that $\hat \tau_2 = 2.6$ and our shooting algorithm finds the approximate value $2.6036023.$ GEKKO was tested with up to 1000 time steps, while our shooting algorithm was run with 150 time steps. This indicates that our method is more accurate when it comes to finding switching times and approximating bang-singular solutions. This feature of shooting methods has been already observed in the literature (see {\em e.g.} \cite{Trelat2012}.
\end{remark}
}

\begin{figure}[!h]
\centering
\begin{center}
\subfigure{\includegraphics [width=5.5cm]{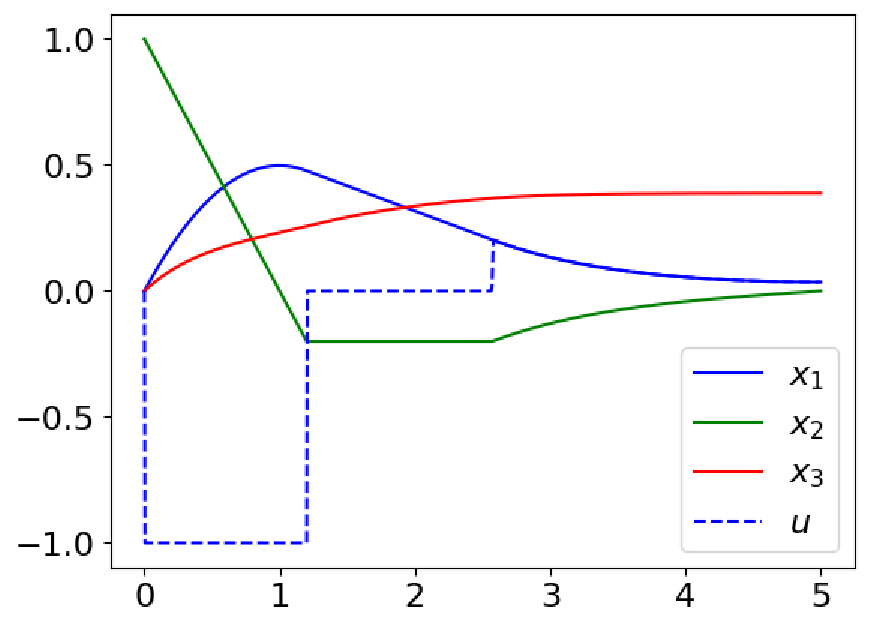}}
\subfigure{\includegraphics [width=5.5cm]{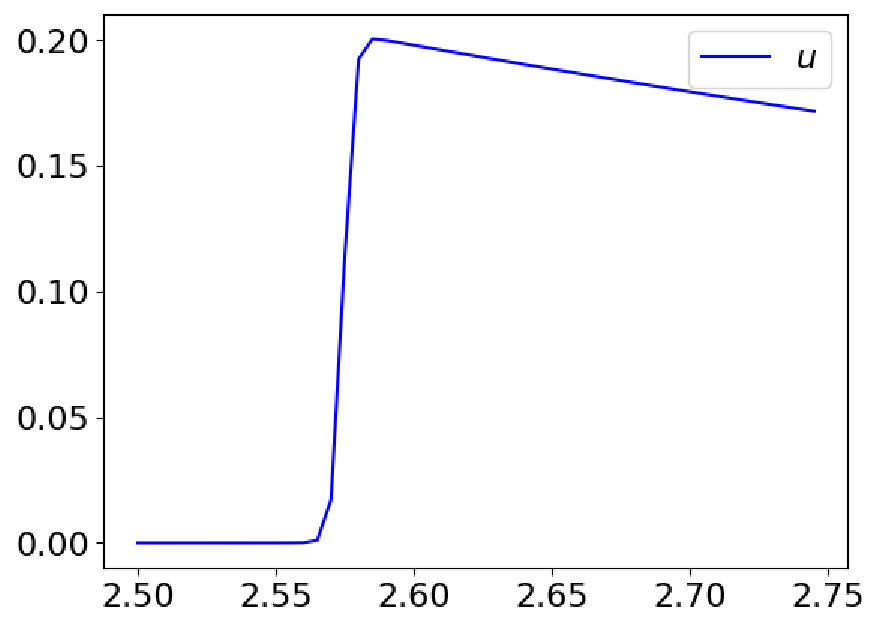}}
\end{center}
\caption{Numerical solution using GEKKO. On the left: optimal control and state variables. On the right: zoom of the optimal control around the second switching time.}
\label{GEKKO2}
\end{figure}

\section{Conclusion}

We have shown that, for problems that are affine w.r.t. the control, and for which the time interval can be partitioned in a finite set of arcs, the shooting algorithm converges locally quadratically if a second order sufficient condition holds, and thus the method may be an efficient way to get a highly accurate solution. 

The essential tool was to formulate a transformed problem. We think that this approach could be extended to more general problems with several state constraints and control variables, the latter possibly entering nonlinearly in the problem.  {Problems with vector control are important in complex models, see {\em e.g.} \cite{LeparouzHerisseJean2022,LeparouzHerisseJean2022CDC} where they studied a control-affine problem with vector control, state constraints, but no singular arcs, generically, in the state-constrained solution. When considering several controls and state constraints, it becomes} important to take into account the possibility of having high-order state constraints.

\section*{Acknowledgements}

The first author was supported by FAPERJ (Brazil) through the {\em Jovem Cientista do Nosso Estado} Program and by CNPq (Brazil) through the {\em Universal} Program and the
Productivity  in Research Scholarship.
  The second author was supported by the FiME Lab Research Initiative (Institut Europlace de Finance), and by the PGMO program.
  
  {We acknowledge the anonymous reviewers for their careful reading and comments that helped us improve this manuscript. 
  
  The first author thanks Gabriel de Lima Monteiro  for his suggestions on the numerical implementations.}

  \section*{Statements and Declarations}
  
  The authors have no conflicts of interest to declare that are relevant to the content of this article.


\bibliographystyle{plain}
\bibliography{shooting_biblio}



\end{document}